\documentclass{llncs}
\usepackage{amssymb,graphicx,color,enumerate}
\usepackage{amsmath,framed,cite}

\title{Tuple domination on graphs with the \\ consecutive-zeros property\thanks{Partially supported by grants PICT ANPCyT 0410 (2017-2020) and 1ING631 (2018-2020)}}

\author{M.P. Dobson\inst{1}
\and
V. Leoni\inst{1,2}
\and
M. I. Lopez Pujato\inst{1,2}
\institute{FCEIA, Universidad Nacional de Rosario, Rosario, Santa Fe, Argentina\\
\email{pdobson@fceia.unr.edu.ar}
\and
CONICET, Argentina\\
\email{valeoni@fceia.unr.edu.ar, lpujato@fceia.unr.edu.ar}
}}

\begin{document}
\maketitle

\begin{abstract} 
 The \emph{$k$-tuple domination problem}, for a fixed positive integer $k$, is to find a minimum sized vertex subset such that every vertex in the graph is dominated by at least $k$ vertices in this set. 
The $k$-tuple domination is NP-hard even for chordal graphs. For the class of circular-arc graphs, its  complexity remains open for $k\geq 2$. A $0,1$-matrix has the \emph{consecutive 0's property (C0P) for columns} if there is a permutation of its rows that places the 0's consecutively in every column. 
Due to A. Tucker, graphs whose augmented adjancency matrix has the C0P for columns are circular-arc.  In this work we study  the $k$-tuple domination problem on graphs $G$ whose augmented adjacency matrix has the C0P for columns, for $ 2\leq k\leq |U|+3$, where $U$ is the set of universal vertices of $G$. From an algorithmic point of view, this takes linear time.


\end{abstract}

\section{Preliminaries, definitions and notation}\label{def}

In this work we consider finite simple graphs $G$, where $V(G)$ and $E(G)$ denote its vertex and edge sets, respectively.
$G'$ is a (vertex) \emph{induced subgraph} of $G$ and write $G'\subseteq G$, if $E(G')=\{uv: uv\in E(G), \{u,v\}\subseteq  V'\}$, for some $V'\subseteq  V(G)$. When neccesary, we use $G[V']$ to denote $G'$.
Given $S\subseteq  V(G)$, the induced subgraph $G[V(G)\setminus S]$ is denoted by $G-S$.
For simplicity, we write $G-v$ instead of $G-\{v\}$, for $v\in V(G)$. 

The \emph{(closed) neighborhood} of $v\in V(G)$ is $N_G[v] = N_G(v) \cup \{v\}$, where $N_G(v)=\{u\in V(G): uv \in E(G) \}$. The \emph{minimum degree} of $G$ is denoted by $\delta(G)$ and is the minimum between the cardinalities of $N_G(v)$ for all $v$.

A vertex $v \in V(G)$ is \emph{universal} if $N_G[v]=V(G)$.

A \emph{clique} in $G$ is a subset of pairwise adjacent vertices in $G$.

 A \emph{stable set} in $G$ is a subset of mutually nonadjacent vertices in $G$ and the cardinality of a stable set of maximum cardinality in $G$ is denoted by $\alpha(G)$. 
 
 A graph $G$ is \emph{circular-arc} if  it has an intersection model consisting of arcs  on a circle, that is, if there is a one-to-one correspondence between the vertices of $G$ and a family of arcs on a circle such that two distinct vertices are adjacent in $G$ when the corresponding arcs intersect.
A graph $G$ is an \emph{interval graph} if it has an intersection model consisting of intervals on the real line, that is, if there exists a family ${\cal I}$ of intervals on the real line and a one-to-one correspondence between the vertices of $G$ and the intervals of ${\cal I}$ such that two vertices are adjacent  in $G$ when the corresponding intervals intersect. A \emph{proper interval graph} is an interval graph that has a \emph{proper interval model}, that is, an intersection model in which no interval contains another one. 
Circular-arc graphs constitute a superclass of proper interval graphs and they are of interest to workers in coding theory because of their relation to ``circular'' codes.   

\bigskip

$J$ denotes the square matrix whose entries are all 1's and  $I$ the identity matrix, both of appropriate sizes.

Associated with a graph $G$ is the \emph{adjacency matrix} $M(G)$  defined with entry $m_{ij}=1$ if vertices $v_i$ and $v_j$ are adjacent, and $m_{ij}=0$ otherwise. Note that $M(G)$ is symmetric and has $0's$ on the main diagonal. The  \emph{augmented adjacency matrix} or \emph{neighborhood matrix} $M^*(G)$ with entries $m^*_{ij}$ is defined as $M^*(G):=M(G) + I $, i.e.  $M(G)$ with $1's$ added on the main diagonal.

A $0,1$-matrix has the \emph{consecutive 0's} property (C0P) for columns if there is a permutation of its rows that places the 0's consecutively in every column. This property was presented by Tucker in~\cite{Tucker}.

\begin{figure}[ht]\label{ejemplo1}
\centering
\includegraphics[scale=1.1 ]{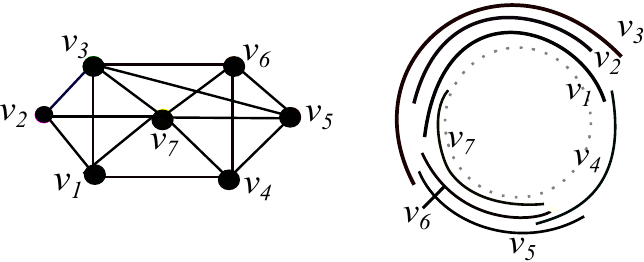}
 
\caption{A graph $G$ with the C0P for columns  and  a circular-arc model for $G$.}
\end{figure}

Fulkerson and Gross \cite {fulkerson} have described an efficient algorithm to test whether a $0,1$-matrix has the C0P for columns and to obtain a desired row permutation when one exists.

\begin{figure}[ht]\label{ejemplo2}
\centering

 \; \; \; $M^*(G)=\left(
\begin{array}{ccccccccccccccccc}
1 && 1 && 1 && 1 && 0 && 0 && 1\\
1 && 1 && 1 && 0 && 0 && 0 && 1\\
1 && 1 && 1 && 0 && 1 && 1 && 1\\
1 && 0 && 0 && 1 && 1 && 1 && 1\\
0 && 0 && 1 && 1 && 1 && 1 && 1\\
0 && 0 && 1 && 1 && 1 && 1 && 1\\
1 && 1 && 1 && 1 && 1 && 1 && 1\\

\end{array}
\right)$.
 
\caption{The augmented adjacency  matrix for graph $G$ in Figure \ref{ejemplo1}.}
\end{figure}

\bigskip

For a nonnegative integer $k$, $D\subseteq V(G)$ is a \emph{$k$-tuple dominating set} of $G$ if $|N_G[v]\cap D|\geq k$, for every $v\in V(G)$. Notice that $G$ has $k$-tuple dominating sets if and only if $k \leq \delta(G)+1$ and, if  $G$ has a $k$-tuple dominating set $D$, then $|D|\geq k$. When $k\leq \delta(G)+1$, $\gamma_{\times k}(G)$ denotes the  cardinality of a  $k$-tuple dominating set of $G$ of minimum size and  $\gamma_{\times k}(G)=+\infty$, when $k> \delta(G)+1$. $\gamma_{\times k}(G)$ is called the \emph{$k$-tuple dominating number} of $G$. Observe that  $\gamma_{\times 1}(G)= \gamma(G)$, the usual domination number. Besides, note that $\gamma_{\times 0}(G)=0$  for every graph $G$~\cite{HararyH}. When $G$ is not connected, the $k$-tuple dominating number of $G$ is defined as the sum of $k$-tuple dominating numbers of its connected component.

For a graph $G$, a positive integer $t$  and  $S\subseteq V(G)$ with $t\leq |S|$, we say that $S$   \emph{$t$-dominates} $G$ if  $S$ is a $t$-tuple dominating set of $G$.

\bigskip

Corcerning computational complexity results, the decision problem (fixed $k$) associated with this concept is NP-complete even for chordal graphs~\cite{DNL2011}.  
It is natural then try to find subclasses of chordal graphs where these problems are ``tractable''.

Efficient algorithms for the problem corresponding to $k=1$ (the usual domination problem) are alrealdy presented in ~\cite{Hsu} and ~\cite{chang} for any circular-arc graph. 
Besides, among the known polynomial time solvable instances of the problem for the case $k=2$, proper interval graphs  constitute the maximal subclass of chordal graphs already studied~\cite{2TDinterval}. Proper interval were characterized by Roberts~\cite{Roberts68} as those graphs whose augmented adjacency matrices have the  consecutive 1's property for columns (defined also by Tucker~\cite{Tucker} in a similar way as the C0P property).

With a different approach,  polynomial algorithms were recently provided for some variations of domination, say $k$-domination and  total $k$-domination (for fixed $k$) for proper interval graphs~\cite{LP1}.

The slightly diference involved in $k$-domination, $k$-tuple domination and total $k$-domination problems makes them useful in various applications, for example in forming sets of representatives or in resource allocation in distributed computing systems. However, the problems are all known to be NP-hard and also hard to approximate~\cite{MR3027964}.

\medskip

In this work we study $2$- and $3$-tuple domination on the subclass of circular-arc graphs that have the C0P for columns. 
Our results allow to solve the $k$-tuple domination problem in this class for $2\leq k \leq |U|+3$.  
In Sections \ref{seccion2} and \ref{seccion3}, we present some special properties on $k$-tuple domination for any positive integer $k$. The study of the problem for $k=2$ and $k=3$ is developed in Section \ref{seccion4} and further analysis for the general case is given in Section \ref{seccion5}. 

\section{$k$-tuple dominating sets on graphs with universal vertices}\label{seccion2}

From the definition, it is clear that $\gamma_{\times k}(G)\geq k$ for every graph $G$ and positive integer $k$. Besides, it is remarkable  that  $S\subseteq V(G)$  $|S|$-dominates $G$ if and only if each vertex of $S$ is a universal vertex. Then,

\begin{lemma}\label{Umasdek}
Let $G$ be any graph, $U$ the set of its  universal vertices and $k$ a positive integer. Then  $\gamma_{\times k}(G)=k$ if and only if $\left|U\right|\geq k$.
\end{lemma}

Notice that, when $u$ is  a universal vertex of $G$ and $D$ is a $k$-tuple dominating set of $G$ with $u$ not in $D$, then by interchanging $u$ with any vertex of $D$, we obtain another $k$-tuple dominating set containing $u$. Formally,

\begin{remark}\label{universal}
If $G$ is  a graph and $u$ a universal vertex of $G$, there exists  a  $k$-tuple dominating set $D$ of $G$ such that $u\in D$. 
\end{remark}

From this remark, it is easy to prove the following relationship:

\begin{proposition}\label{masuno1} Let  $G$ be a graph, $u$ a universal vertex of $G$ and  $k$ a  positive integer. Then $$\gamma_{\times k}(G)=\gamma_{\times (k-1)}(G-u)+1.$$
\end{proposition}

\begin{proof}

Let $D$ be a $k$-tuple dominating set of $G$ with $\left|D\right|=\gamma_{\times k}(G)$. 

If  $u\in D$, then $D-u$ is a $(k-1)$-tuple dominating set of $G-u$, thus $\gamma_{\times(k-1)}(G-u)+1\leq \left|D\right|=\gamma_{\times k}(G)$.
If $u\notin D$, from Remark \ref{universal} we can build  a $k$-tuple dominating set $D'$  of $G$ with $\left|D'\right|=\gamma_{\times k}(G)$ and  $u\in D'$ and proceed as above with $D'$ instead of $D$.

On the other side, let $D$ be a minimum $(k-1)$-tuple dominating  set of $G-u$. It is clear that  $D\cup\left\{u\right\}$ is a $k$-tuple dominating set of $G$ since  $u$ is a universal vertex.
Then $\gamma_{\times k}(G)\leq\left|D\cup\left\{u\right\}\right|=\left|D\right|+1=\gamma_{\times(k-1)}(G-u)+1$ and the proof is complete.
\qed\end{proof}

The above lemma can be generalized as follows:

\begin{lemma}\label{masuno2}
Let  $G$ be a graph,  $U$ the set of its  universal vertices and $k$ a positive integer with $\left|U\right|\leq k-1$. Then   $$\gamma_{\times k}(G)=\gamma_{\times (k-\left|U\right|)}(G-U)+\left|U\right|.$$
\end{lemma}

It is clear from Lemmas \ref{Umasdek} and \ref{masuno2} the following

\begin{corollary}
For a graph $G$ and $U\neq \emptyset$  the set of its  universal vertices,  if $\gamma_{\times i} (G)$ can be found in polynomial time for $i=1,2,3$, then $\gamma_{\times k} (G)$ can be found in polynomial time for every $k$ with $1 \leq k \leq |U|+3$.
\end{corollary}

\section{$k$-tuple domination and C0P-graphs}\label{seccion3}

Recall that a $0,1$-matrix has the  C0P for columns if there is a permutation of its rows that places the 0's consecutively in every column. 
We introduce the following definition:

\begin{definition}
A graph $G$ whose augmented adjancency matrix, $M^{*}(G)$, has the C0P for columns is called a {\emph{C0P}-graph}.
\end{definition}

\begin{remark} 
It is clear that if $G$ is a COP-graph then $G - U$ is a COP-graph.
\end{remark}

\medskip

Let $G$ be a C0P-graph with  its vertices  indexed so that the $0$'s occur consecutively in each column of $M^*(G)$. Let  $C_1$ be the set of columns whose  0's are below  the main diagonal, $C_2$ the set of columns whose  0's are above  the main diagonal, and $U$ the set of columns without 0's. Sets $C_1$, $C_2$ and $U$ partition $V(G)$, $G[C_1]$ and $G[C_2]$ are cliques in $G$ and $U$ is the set of universal vertices of $G$. We denote this partition by $(C_1, C_2, U)$, or simply $(C_1, C_2)$ when $U=\emptyset$. In the later case, $|C_1|\geq 2$ and $|C_2|\geq 2$.  Also for simplicity, we denote $G_1:=G[C_1]$ and $G_2:=G[C_2]$.

From now on, $G$ is a C0P-graph and $(C_1, C_2, U)$ is the above mentioned partition of $V(G)$.

\bigskip

It is easy to prove the following upper bound on the size of  a mimimun $k$-tuple dominating set of a C0P-graph:

\begin{lemma} \label{contenidouniversal}
Let $G$ be a C0P-graph and $k$ a positive integer. If $|C_i|\geq k$  for $i=1,2$, then $$\gamma_{\times  k}(G)\leq 2k.$$
\end{lemma}

\begin{proof}
Let $D_i\subseteq C_i$  with  $|D_i|=k$, for $i=1,2$ and consider the set $D_1\cup D_2$. Take  $v\in V(G)$. If $v\in C_i$, then $D_i \subseteq  N_G[v]$, thus  $|N_G[v] \cap (D_1\cup D_2)| \geq |D_i| =k$, for $i=1,2$. If $v\in U$, clearly $N_G[v]= C_1 \cup C_2$, thus  $|N_G[v] \cap (D_1\cup D_2)| = |D_1 \cup D_2| =2k \geq k$. Thus $D_1\cup D_2$ is a $k$-tuple dominating set of $G$ and the upper bound follows.
\qed\end{proof}

\bigskip

Lemmas \ref{masuno2} and  \ref{contenidouniversal}  allow us to restrict our study of  C0P-graphs to those  with partition $(C_1, C_2, U)$, where $U=\emptyset$ and $C_1$ and $C_2$ are nonempty sets. Under these assumptions, we have $k+1 \leq\gamma_{ k}(G)\leq 2k$, for any C0P-graph $G$.

For a given C0P-graph $G$ with partition $(C_1, C_2)$, let us denote $V(G)=\{v_1,v_2,\cdots,v_n\}$, $C_1=\{v_1,v_2,\cdots,v_r\}$ and 
$C_2=\{v_{r+1},v_{r+2},\cdots,v_n\}$.
Also let us denote by $M^*_{C_iC_j}$,  the submatrix of $M^{*}(G)$ with rows indexed by $C_i$ and columns by $C_j$. Notice that $M^*_{C_1C_1}=M^*_{C_2C_2}=J$.

\begin{figure}[ht]\label{esquema}
\begin{center}
\includegraphics[scale=1.2]{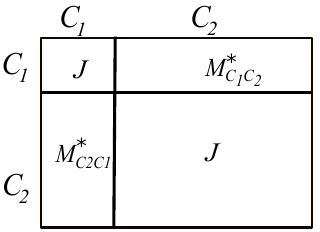}

\caption{Augmented adjacency matrix $M^*(G)$ of a C0P-graph $G$ with $U=\emptyset$.}
\end{center}
\end{figure}

\subsection{Construction of auxiliary interval graphs $H_i$}

Let $G$ be a C0P-graph and $(C_1, C_2)$  the above mentioned partition of $V(G)$.

We construct two interval graphs $H_1$ and $H_2$ with $V(H_i)=C_i$ for $i=1,2$ in the following way: 

\begin{itemize}
\item for each vertex $v_i\in C_1$,  define an interval $I_i$ from $\left[r+1,n\right]_{\mathbb{N}}$ such that, if the  consecutive $0$'s of column $v_i$ correspond to the vertices $v_p,...,v_{p+s}$ where $p\geq r+1$ and  $p+s\leq n$, then $I_i=\left[p,p+s\right]_{\mathbb{N}}$.

\item for each vertex $v_i \in C_2$,  define an interval $I_i$ from $\left[1,r\right]_{\mathbb{N}}$ such that, if the consecutive $0$'s of column $v_i$ correspond to the vertices $v_p,...,v_{p+s}$ with $p\geq 1$ and $p+s\leq r$, then $I_i=\left[p,p+s\right]_{\mathbb{N}}$. 
\end{itemize}

We will say that vertex $v_i$ is represented by the interval $I_i$, $\forall i=1,...,n$. 

The two interval graphs $H_1$ and $H_2$ constructed as above have interval models $\mathcal{I}_1=\left\{I_1,I_2,...,I_r\right\}$ and $\mathcal{I}_2=\left\{I_{r+1},I_{r+2},...,I_n\right\}$, respectively.

\bigskip
\begin{figure}[ht]\label{auxiliary}
\centering
\includegraphics[scale=1.2]{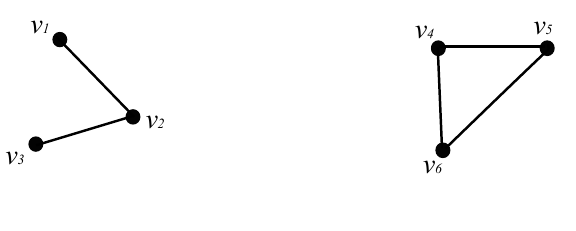}
\caption{Graphs $H_1$ and $H_2$ related to graph $G$ of Figure \ref{ejemplo1}.}
\end{figure}

\bigskip

\begin{remark} For a C0P-graph with partition $(C_1, C_2, U)$ with $U \neq \emptyset$, graphs $H_1$ and $H_2$ are defined as above from the subgraph $G-U$ of $G$.
\end{remark}

It is clear that given two intersecting intervals $I_i$ and $I_j$ of $H_1$ for  $1\leq i \neq j \leq p$, there exists  $q$ with $p+1\leq q\leq n$ such that $m^*_{qi}=m^*_{qj}=0$. This means that $v_qv_i\notin E(G)$ and $v_qv_j\notin E(G)$. 
In other words, given two non intersecting intervals $I_i$ and $I_j$ of $H_1$  for  $1\leq i \neq j \leq p$, we have $m^*_{qi}=1$ or $m^*_{qj}=1$  for all $q$ with $p+1\leq q\leq n$. Therefore in  each file of $M^*_{C_2C_1}$ there exist at least one  $1$ in the columns corresponding to vertex  $v_i$ or $v_j$ and then $v_qv_i\in E(G)$ or $v_q v_j\in E(G)$  for all $q$ with $p+1\leq q\leq n$. 

In a similar way, the above argument clearly holds  for the interval graph $H_2$.

\bigskip
\subsection{Stables sets of $H_i$ and tuple dominating sets of $G$}

We will denote by $\alpha_i$ the stability number of the inteval graphs $H_i$ constructed as in the previous subsection, for $i=1,2$. Let us remark that the stability number of an interval graph can be found in linear time \cite{gavril}.

The following fact easily follows:

\begin{lemma}\label{alosumo}  Let $G$ be a C0P-graph with partition $(C_1,C_2)$, $S\subseteq C_j$ and $t$ a positive integer such that $S$ $t$-dominates $G_i$  for $i\neq j$. Then $ |S|  \geq t+1$.
\end{lemma}

\begin{proof}
Since $U=\emptyset$, the proof easily follows from the fact that for each vertex $v\in C_i$, there is a non adjacent vertex $w\in C_j$, for $i\neq j$ and $i=1,2$. 
\qed\end{proof}

It is straigthforward that any subset $S\subseteq C_i$ $|S|$-dominates $G_i$, for each $i=1,2$ and  at most $(|S|-1)$-dominates the whole graph $G$. When considering stable sets of $H_i$, the interesting fact is the following, which will be the key of the results in the next section: 

\begin{proposition}\label{estable} Let $G$ be a C0P-graph with partition $(C_1,C_2)$ and $S\subseteq {C_i}$, with $i=1, 2$. Then $S$  $(|S|-1)$-dominates $G_j$  ($i\neq j$) if and only if $S$ is a stable set of $H_i$. 
\end{proposition}

\begin{proof}
 $S$ is an $(|S|-1)$-tuple dominating set of  $G_j$ if and only if for  every vertex $v\in C_j$, $|N_{G_j}[v]\cap S|\geq |S|-1$. In other words, $S$ is an $(|S|-1)$-tuple dominating set of  $G_j$ if and only if  for each file of $M^*_{C_jC_i}$ there exists at most one zero in  the columns  corresponding to  vertices in $S$. This is equivalent to say that the set of intervals $\{I_t\}_{t:v_t\in S}$ are pairwise non-adjacent, i.e $S$ is a stable set of $H_i$.

\qed\end{proof}

The relationship bewtween tuple dominating sets of $G$ and stable sets of the auxiliary interval graphs $H_1$ and $H_2$ allow us to solve the $2$, $3$-tuple domination problems on C0P-graphs.

\section{ $2$- and $3$-tuple domination for C0P-graphs}\label{seccion4}

\subsection{$2$-tuple domination}

\begin{theorem}\label{2tuple}  Let $G$ be a $C0P$-graph with  partition $(C_1, C_2, U)$  and graphs $H_i$ defined as in the previous section, for $i=1,2$. 
\begin{enumerate}
\item[1)] If $|U|=1$, then $\gamma_{\times 2}(G)=  3$.
\item[2)] If $|U|\geq 2$, then $\gamma_{\times 2}(G)=2$.
\item[3)] If $|U|=0$ and $\alpha_1+\alpha_2\geq 3$  then $\gamma_{\times 2}(G)=3$.
\item[4)] If $|U|=0$ and  $\alpha_1=\alpha_2=1$ then $\gamma_{\times 2}(G)=4$.
\end{enumerate}
\end{theorem}

\begin{proof} 
\begin{enumerate}
\item[1)] Follows from Proposition \ref{masuno1}. 
\item[2)]  Follows from Lemma \ref{Umasdek}.
\item[3)] Suppose $|U|=0$ and $\alpha_1+\alpha_2\geq 3$ and  let $S_1$ and $S_2$ be stable sets of $H_1$ and $H_2$ respectively with $|S_1\cup S_2|=3$. Clearly, $S_i$ $|S_i|$-dominates $G_i$ and also $(|S_i|-1)$-dominates $G_j$  for each $i=1,2$ and $i\neq j$. Thus  $S_1\cup S_2$ is a $2$-tuple dominating set of $G$ and then $\gamma_{\times k}(G) \leq 3$. But  $\gamma_{\times k}(G) > 2$ from Lemma \ref{Umasdek} and thus the result follows. 

\item[4)] Suppose  $|U|=0$ and  $\alpha_1=\alpha_2=1$. Then  $|D\cap C_j|\geq 2$  for $j=1,2$ and every $2$-tuple dominating set $D$ of  $G$.   Thus  $\gamma_{\times  2}(G) \geq 4$. The result then follows from Lemma \ref{contenidouniversal}.
\end{enumerate}
\qed\end{proof}

\begin{corollary}
The 2-tuple domination problem can be solved in linear time on C0P-graphs.
\end{corollary}

\begin{proof} Follows from the fact that finding the stability  number of an interval graph is linear.
\qed\end{proof}

\subsection{$3$-tuple domination}

\begin{theorem}\label{3tuple}  Let $G$ be a $C0P$-graph with  partition $(C_1, C_2, U)$  and graphs $H_i$ defined as in the previous section, for $i=1,2$. 

\begin{enumerate}

\item[i.] If $|U|=1$, then $\gamma_{\times 3}(G)= 4$ if $\alpha_1+\alpha_2\geq 3$, and $\gamma_{\times 3}(G)= 5$ if $\alpha_1+\alpha_2=2$.
\item[ii.] If $|U|=2$, then $\gamma_{\times 3}(G)=4$.
\item[iii.] If $|U|=3$, then $\gamma_{\times 3}(G)= 3$.
\item[iv.] If $|U|=0$ and $\alpha_1+\alpha_2\geq 4 $ then $\gamma_{\times 3}(G)=4$.
\item[v.] If $|U|=0$ and $\alpha_1=\alpha_2=1$ then $\gamma_{\times 3}(G)=6$.
\item[vi.] If $|U|=0$ and $\alpha_1+ \alpha_2=3$ then $\gamma_{\times 3}(G)=5$.
\end{enumerate}
\end{theorem}

\begin{proof}   
\begin{enumerate}
\item[i.] Follows from Proposition \ref{masuno1} and items 3 and 4 of  Theorem \ref{2tuple}.
\item[ii.] Follows from Lemma \ref{masuno2}.
\item[iii.] Follows from Lemma \ref{Umasdek}.
\item[iv.)] Suppose $|U|=0$ and $\alpha_1 +\alpha_2\geq 4$ and  let $S_1$ and $S_2$ be stable sets of $H_1$ and $H_2$ respectively, with $|S_1\cup S_2|=4$. Clearly, $S_i$ $|S_i|$-dominates $G_i$ and also $(|S_i|-1)$-dominates $G_j$ for $i\neq j$ and $i=1,2$. Thus  $S_1\cup S_2$ is a $3$-tuple dominating set of $G$ implying  $\gamma_{\times 3}(G) \leq 4$. But  $\gamma_{\times 3}(G) > 3$ from Lemma \ref{Umasdek} concluding 
$\gamma_{\times k}(G) = 4$. 
\item[v.)] Suppose $|U|=0$ and $\alpha_1=\alpha_2=1$. Then  $|D\cap C_j|\geq 3$  for $j=1,2$ for every $3$-tuple dominating set $D$ of  $G$.   Thus  $\gamma_{\times  2}(G) \geq 6$. The result then holds from Lemma \ref{contenidouniversal} (when $|C_i| \geq 3$; in other case the problem is infeasible).
\item[vi.)] Suppose w.l.o.g. that $\alpha_1=1 \wedge \alpha_2=2$. 
It is not  difficult to see that it is enough to consider the case $|C_1|\geq 2$ and $|C_2|\geq 3$ (in any other case, $\alpha_1= 1$ together with $|C_2|=2$ imply the existence of a vertex $w \in C_2$ not adjacent to every $v\in C_1$,  leading to the infeasibility of the problem).

Let $S_1$ and $S_2$ be stable sets of $H_1$ and $H_2$ respectively where $\left|S_i\right|=\alpha_i$ for $i=1,2$. Then $S_1\cup S_2$ $2$-dominates $G$. Take two vertices $w_1\in C_1-S_1$ and $w_2\in C_2-S_2$, thus the set $S_1\cup S_2\cup \{w_1,w_2\}$ is a $3$-dominating set of $G$ of cardinality $5$, implying $\gamma_{\times 3}(G) \leq 5$. 

Now, since $\gamma_{\times 3}(G) \geq 4$ ($U=\emptyset$), it is enough to show that $\gamma_{\times k}(G)\neq 4$. Suppose  $D$ is a minimum $3$-tuple dominating set of $G$ with $|D|=4$ and denote $d_1=|D\cap C_1|$ and $d_2=|D\cap C_2|$. 
Consider $t_i=$ max $\{|N[x] \cap D_i|: x\in C_j\}$, for $i=1,2$, $i\neq j$. From Lemma \ref{alosumo},  $t_i\leq d_i-1$ for $i=1,2$, and moreover, $t_1+t_2\geq 3$ (otherwise, for each $x\in V$, $|N[x]\cap D|=|N[x]\cap D_1|+|N[x]\cap D_2|\leq t_1+t_2 <3$,  contradiction). 
Then $d_1+d_2-1=3\leq t_1+t_2\leq d_1+d_2-2=2$, which leads to a  contradiction. Thus, we have the desired equality.

\end{enumerate}

\qed\end{proof}

\begin{corollary}
The 3-tuple domination problem can be solved in linear time on C0P-graphs.
\end{corollary}

\begin{proof} Follows from the fact that finding the stability  number of an interval graph can be done in linear time.
\qed\end{proof}

\begin{example}
Recall graph $G$ from Figure \ref{ejemplo1} and the auxiliary graphs $H_1$ and $H_2$ of Figure \ref{auxiliary}. The results exposed in this section can be applied approprietly in order to calculate the values of $\gamma_{\times i}(G)$ for $i=1,2,3,4$. Actually,  since $\alpha_1=2$ and $\alpha_2=1$, we have:

$$\gamma_{\times 4}(G)=\gamma_{\times 3}(G - v_7) +1= 5+1=6,$$

$$\gamma_{\times 3}(G)=\gamma_{\times 2}(G - v_7) +1= 3+1=4,$$

$$\gamma_{\times 2}(G)=\gamma_{\times 1}(G - v_7) +1= 2+1=3$$ and

$$\gamma_{\times 1}(G)= 1.$$

\bigskip

\bigskip

\bigskip

\begin{figure}[ht]\label{ejemplo2}
\centering
\includegraphics[scale=1.1 ]{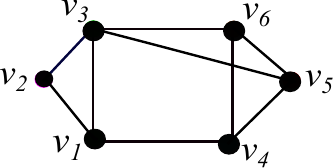}

\caption{Graph $G- U$, where $G$ is the graph of Example \ref{ejemplo1} and $U=\{v_7\}$.}
\end{figure}
\end{example}

\section{Further analysis for any $k$}\label{seccion5}

Some of the results in the previous section can be generalized in the following way:

\begin{proposition}\label{alfas}
Let $G$ be a $C0P$-graph with partition $(C_1, C_2)$  and graphs $H_i$ defined as in the previous section, for $i=1,2$.

\begin{enumerate}
\item if $\alpha_i=1$ and $D$ is a  $k$-tuple dominating set $D$ of  $G$, then  $|D\cap C_j|\geq k$  with $1 \leq i\neq j \leq2$;
\item if $\alpha_1+ \alpha_2=2$ then $\gamma_{\times k}(G)=2k$;             
\item if $\alpha_1+ \alpha_2>k$ then $\gamma_{\times k}(G)=k+1$. 
\item if $\alpha_1+ \alpha_2= k$ and $|C_i|\geq \alpha_i+1$ for $i=1,2$ then $\gamma_{\times k}(G)=k+2$. 
\end{enumerate}
\end{proposition}

\begin{proof}
\begin{enumerate}
\item  W.l.o.g., asume $i=1$. Then  $\alpha_1=1$ implies that there exists $j$ with  $r+1 \leq j\leq n$ such that $m^*_{ji}=0$ for all $i$ with $1 \leq i\leq r$. This means that  vertex $v_j$ ($v_j\in C_2$) is non adjacent to every vertex of $C_1$, thus  $|D\cap C_2|\geq k$ for each  $k$-tuple dominating set  $D$ of  $G$.

\item If $\alpha_1=1=\alpha_2$, then previous item implies that any $k$-tuple dominating set of  $G$ has at leat $2k$ vertices.   Thus  $\gamma_{\times  k}(G) \geq 2k$. The equality holds from Lemma \ref{contenidouniversal}.

\item Let $S_1$ and $S_2$ be stable sets of $H_1$ and $H_2$ respectively, with $|S_1\cup S_2|=k+1$. Clearly, $S_i$ $|S_i|$-dominates $G_i$ and also $(|S_i|-1)$-dominates $G_j$  for $i=1,2$ and  $i\neq j$. Thus  $S_1\cup S_2$ is a $k$-tuple dominating set of $G$ and then  $\gamma_{\times k}(G) \leq k+1$. But  $\gamma_{\times k}(G) > k$ from Lemma \ref{Umasdek} and then $\gamma_{\times k}(G) = k+1$. 

\item  Let $S_1$ and $S_2$ be maximum stable sets of $H_1$ and $H_2$ respectively. It is clear that $S_1\cup S_2$ is a $(\alpha_1+\alpha_2-1)$-dominating set of $G$, i.e a $(k-1)$-dominating set of $G$. Take two vertices $w_1\in C_1-S_1$ and $w_2\in C_2-S_2$. Thus the set $S_1\cup S_2\cup \{w_1,w_2\}$ is a $k$-tuple dominating set of $G$ of cardinality $k+2$, implying $\gamma_{\times k}(G)\leq k+2$. 

Now, since $\gamma_{\times k}(G) \geq k+1$ ($U=\emptyset$), it suffies to show that $\gamma_{\times k}(G)\neq k+1$. Suppose  $D$ is a minimum $k$-tuple dominating set of $G$ with $|D|=k+1$ and denote $d_1=|D\cap C_1|$ and $d_2=|D\cap C_2|$. 
Consider $t_i=$ max $\{|N[x] \cap D_i|: x\in C_j\}$, for $i=1,2$, $i\neq j$. From Lemma \ref{alosumo},  $t_i\leq d_i-1$ for $i=1,2$, and moreover, $t_1+t_2\geq k$ (otherwise, for each $x\in V$, $|N[x]\cap D|=|N[x]\cap D_1|+|N[x]\cap D_2|\leq t_1+t_2 <k$, a contradiction). 
Then $d_1+d_2-1=k\leq t_1+t_2\leq d_1+d_2-2$, whcih leads to a  contradiction. Thus, we have the desired equality.
\end{enumerate}
\qed\end{proof}

\section{Conclusions}

In this work,  we solved in linear time the $k$-tuple domination problem on a subclass of circular-arc graphs, called C0P-graphs for $ 2\leq k\leq |U|+3$, where $U$ is the set of universal vertices of the input graph $G$. We think that ---under a suitable implementation--- the thecniques used in this paper together with the more general result in Theorem \ref{alfas} can be further developed to solve the  problem for the remaining values of $k$, even for other subclasses or moreover, the whole class of circular-arc graphs where the problems remain unsolved.

\end{document}